\documentclass[4 pt, a4paper]{article}
\usepackage{latexsym}
\usepackage{floatpag}
\usepackage{multirow}
\usepackage{graphicx}
\usepackage{amsmath}
\usepackage{amssymb}
\usepackage{centernot}
\usepackage{amsfonts}
\usepackage{enumerate}
\usepackage[a4paper,left=2cm,right=2cm,top=2cm,bottom=2cm]{geometry}
\usepackage{diagbox}
\usepackage{authblk}
\usepackage{amsthm}
\usepackage{hyperref}
\makeatletter
\def\th@plain{%
  \thm@notefont{}% same as heading font
  \itshape % body font
}
\def\th@definition{%
  \thm@notefont{}% same as heading font
  \normalfont % body font
}
\makeatother
\begin{document}
\title{The Number of Locally $p$-stable Functions on $Q_n$}
\author{Asier Calbet}
\affil{School of Mathematical Sciences \\ Queen Mary, University of London \\ Mile End Road, London E1 4NS \\ United Kingdom
\\ \href{mailto:a.calbetripodas@qmul.ac.uk}{\nolinkurl{a.calbetripodas@qmul.ac.uk}}}
\date{}
\maketitle

\newtheorem{theorem}{Theorem}
\newtheorem{lemma}{Lemma}
\newtheorem{definition}{Definition}
\newtheorem*{thr}{Theorem \ref{est.}}
\newtheorem*{theo}{Theorem \ref{stab.}}
\newtheorem{prop}{Proposition}
\newtheorem{q}{Question}
\newtheorem{cor}{Corollary}
\newtheorem{obs}{Observation}

\begin{abstract}
\noindent A Boolean function $f:V(G) \to \{-1,1\}$ on the vertex set of a graph $G$ is \emph{locally $p$-stable} if for every vertex $v$ the proportion of neighbours $w$ of $v$ with $f(v)=f(w)$ is exactly $p$. This notion was introduced by Gross and Grupel in \cite{Reference 1} while studying the scenery reconstruction problem. They give an exponential type lower bound for the number of isomorphism classes of locally $p$-stable functions when $G=Q_n$ is the $n$-dimensional Boolean hypercube and ask for more precise estimates. In this paper we provide such estimates by improving the lower bound to a double exponential type lower bound and finding a matching upper bound. We also show that for a fixed $k$ and increasing $n$, the number of isomorphism classes of locally $(1-k/n)$-stable functions on $Q_n$ is eventually constant. The proofs use the Fourier decomposition of functions on the Boolean hypercube.
\end{abstract}

\section{Introduction}
Let $G$ be a graph with vertex set $V(G)$. By a \emph{Boolean function} on $G$ we mean a function $f: V(G) \to \{-1,1\}$. Motivated by the scenery reconstruction problem, Gross and Grupel, in \cite{Reference 1}, define a \emph{locally $p$-stable} function on $G$ to be a Boolean function $f$ on $G$ such that for every vertex $v \in V(G)$ we have

$$ \frac{|\{w \in \Gamma(v) : f(v) = f(w)\}|}{d(v)} = p \  , $$

where $\Gamma(v)$ denotes the neighbourhood of $v$ in $G$ and $d(v)=|\Gamma(v)|$ is the degree of $v$. They say that two Boolean functions $f$ and $g$ on $G$ are \emph{isomorphic} if there is an automorphism $\phi$ of $G$ such that $f=g \circ \phi$. They show that the scenery reconstruction problem on the $n$-dimensional Boolean hypercube is impossible for $n \geq 4$ by constructing two non-isomorphic locally $p$-stable functions and noting that the scenery processes for these functions have exactly the same distribution. \newline

Let us now restrict ourselves to the case when $G=Q_n$ is the $n$-dimensional Boolean hypercube. It will be more convenient for us to work with a re-parametrised definition of locally $p$-stable functions.

\begin{definition}
A \emph{$k$-function} is a Boolean function $f$ on $Q_n$ such that for every vertex $v \in V(G)$ we have

$$ |\{w \in \Gamma(v) : f(v) \neq f(w)\}| = k  . $$
\end{definition}

Note that a $k$-function is precisely a locally $p$-stable function on $Q_n$ with $p=1-k/n$. Combinatorially, a $k$-function corresponds to a partition of $Q_n$ into two parts such that every vertex has precisely $k$ neighbours in the opposite part. \newline

It will also be more convenient for us to work with an extended notion of isomorphism. We say that two real-valued functions $f$ and $g$ on $Q_n$ are \emph{isomorphic} if there is an automorphism  $\phi$ of $Q_n$ and a sign $\epsilon \in \{-1,1\}$ such that $f=\epsilon g \circ \phi$. Note that if $f$ and $g$ are isomorphic then $f$ is a $k$-function if and only if $g$ is. The number of isomorphism classes of $k$-functions changes by at most a factor of $2$ when passing from our definition of isomorphism to that of Gross and Grupel. \newline   

Let us now introduce some notation. For integers $0 \leq k \leq n$, let $F(n,k)$ denote the number of $k$-functions on $Q_n$ and let $G(n,k)$ denote the number of isomorphism classes of 
$k$-functions on $Q_n$. We are mainly interested in $G(n,k)$, but will need $F(n,k)$ in the proofs. \newline

In \cite{Reference 1}, Gross and Grupel obtain a lower bound of the form

$$ G(n,k) = 2^{\Omega(\sqrt{k})}  $$

for $n \geq 2k-2$ and ask for more precise estimates (Question 5.12.). \newline

In this paper, we provide such estimates:

\begin{theorem}\label{est.}

Let $0 \leq k \leq n$ be integers. Then

$$2^{2^{m+o(1)}} \leq G(n,k) \leq 2^{2^{2m+O\left(\log_2 m\right)}} , $$

where $m=min(k,n-k)$.

\end{theorem}

We also prove the following theorem, which is a key ingredient in the proof of Theorem \ref{est.}:

\begin{theorem}\label{stab.}
Let $k \geq 0$ be an integer. Then the sequence $(G(n,k))_{n=k}^{\infty}$ is increasing and eventually constant. Moreover, denoting by n(k) the first value of $n$ after which the sequence is constant, we have 

$$ 3 \cdot 2^{k-1} - 2 \leq n(k) \leq 4.394 \cdot 2^k . $$

\end{theorem}

The paper is organised as follows. In Section \ref{prel.} we introduce some definitions and notation, describe the automorphisms of $Q_n$ and recall some basic facts about Fourier analysis on the Boolean hypercube. In Section \ref{res.} we prove Theorems \ref{est.} and \ref{stab.}.   \newline

\section{Preliminaries}\label{prel.}
In this section we introduce some definitions and notation, describe the automorphisms of $Q_n$ and recall some basic facts about Fourier analysis on the Boolean hypercube.

\subsection{The Boolean hypercube} 
We first introduce some definitions and notation that we will need later on. It will be convenient to think of $Q_n$ as having vertex set $V(Q_n)=\{-1,1\}^n$. The edge set $E(Q_n)$ is the set of pairs of vectors differing in precisely one entry. We will sometimes write $Q_S$, where $S$ is a finite set, for the Boolean hypercube indexed by $S$ (so $V(Q_n)=\{-1,1\}^S$ and $E(Q_n)$ is as before). One can think of $Q_n$ as $Q_{[n]}$, where $[n]=\{1,2,3,\cdots,n\}$. We will write vectors $x \in Q_n$ as $x=(x_1, x_2, x_3, \cdots, x_n)$, so that $x_i \in \{-1,1\}$ for all $i \in [n]$. \newline

For each $x \in Q_n$, we write $\Gamma(x) = \{y \in Q_n : xy \in E(Q_n)\}$ for the neighbourhood of $x$. Let $f$ be a real-valued function on $Q_S$. We say that an index $i \in S$ is \emph{irrelevant} if the value of $f(x)$ does not depend on the value of $x_i$. Otherwise, we say that $i$ is \emph{relevant}. Given a finite set $T \supseteq S$, we can think of $f$ as a function on $Q_T$ for which all the indices in $T \setminus S$ are irrelevant. Conversely, a real-valued function on $Q_T$ for which all the indices in $T \setminus S$ are irrelevant can be thought of as a function on $Q_S$. \newline      

We now describe the automorphisms of $Q_n$. For each $\alpha \in Q_n$ there is an automorphism of $Q_n$, which we will also denote by $\alpha$, given by $\alpha(x)_i=\alpha_i x_i$ for all $x \in Q_n$ and $i \in [n]$. Let $S_n$ be the set of permutations of $[n]$. For each $\sigma \in S_n$ there is an automorphism of $Q_n$, which we will also denote by  $\sigma$, given by $\sigma(x)_i=x_{\sigma(i)}$ for all $x \in Q_n$ and $i \in [n]$. It is well known that any automorphism $\phi$ of $Q_n$ can be written uniquely as $\phi = \alpha \circ \sigma$ with $\alpha \in Q_n$ and $\sigma \in S_n$. \newline

In particular, there are $2^n n!$ automorphisms of $Q_n$. Hence, since there are $2$ signs, every isomorphism class of $k$-functions on $Q_n$ has size at least $1$ and at most $2^{n+1} n! $ , which gives the following lemma. 

\begin{lemma}\label{F,G Lemma}
Let $0 \leq k \leq n$ be integers. Then

$$ \frac{F(n,k)}{2^{n+1}n!} \leq G(n,k) \leq F(n,k) . $$
\end{lemma}

We will also need the following easy lemma later on. 

\begin{lemma}\label{m,n isom.}
Let $0 \leq m \leq n$ be integers and let $f$ and $g$ be real-valued functions on $Q_m$. Then $f$ and $g$ are isomorphic when thought of as functions on $Q_m$ if and only if they are isomorphic when thought of as functions on $Q_n$. 
\end{lemma} 

\begin{proof}

Suppose $f$ and $g$ are isomorphic when thought of as functions on $Q_m$, say $f = \epsilon g \circ \alpha \circ \sigma$, where $\epsilon \in \{-1,1\}$, $\alpha \in Q_m$ and $\sigma \in S_m$. Let $\beta \in Q_n$ and $\tau \in S_n$ be given by

\begin{align*}
\beta_i =  \begin{cases} \alpha_i & \text{for } i \in [m] \\ 1 & \text{for } i \notin [m]  \end{cases} && \text{and} && \tau(i) =  \begin{cases} \sigma(i) & \text{for } i \in [m] \\ i & \text{for } i \notin [m] \\ \end{cases} \ . 
\end{align*}

Then, when $f$ and $g$ are thought of as functions on $Q_n$, we have $f = \epsilon g \circ \beta \circ \tau$, so $f$ and $g$ are isomorphic when thought of as functions on $Q_n$. \newline

Suppose $f$ and $g$ are isomorphic when thought of as functions on $Q_n$, say $f = \epsilon g \circ \alpha \circ \sigma$, where $\epsilon \in \{-1,1\}$, $\alpha \in Q_n$ and $\sigma \in S_n$. Let $\beta \in Q_m$ be given by $\beta_i = \alpha_i$ for all $i \in [m]$. Let $S,T \subseteq [m]$ be the sets of relevant indices of $f$ and $g$, respectively. Then, by considering the set of relevant indices of $f = \epsilon g \circ \alpha \circ \sigma$, we see that $\sigma(T)=S$. Let $\tau \in S_m$ be any permutation such that $\tau(i)=\sigma(i)$ for all $i \in T$. Then, when $f$ and $g$ are thought of as functions on $Q_m$, we have $f = \epsilon g \circ \beta \circ \tau$, so $f$ and $g$ are isomorphic when thought of as functions on $Q_m$.    

\end{proof}

\subsection{Fourier analysis on the Boolean hypercube}
We now recall some basic facts about Fourier analysis on the Boolean hypercube. For a comprehensive treatment see \cite{Reference 2}. Let $V$ be the vector space of real-valued functions on $Q_n$. For each subset $S \subseteq [n]$, let $\chi_{S} \in V$ be the function given by

$$ \chi_{S} (x) = \prod_{i \in S} x_i .$$

The $\chi_{S}$ form a basis of $V$, so any $f \in V$ can be written uniquely as 

$$ f = \sum_{S \subseteq [n]} \hat{f}(S) \chi_{S} ,$$  

where the $\hat{f}(S)$ are the \emph{Fourier coefficients} of $f$. The function $\hat{f}$ mapping each $S \subseteq [n]$ to its Fourier coefficient $\hat{f}(S)$ is known as the \emph{Fourier transform} of $f$ and this decomposition is known as the \emph{Fourier decomposition}. \newline

For each integer $0 \leq k \leq n$, let $V_k$ be the subspace of $V$ spanned by the $\chi_S$ with $|S|=k$. We will need the following two basic facts later on, so we state them here as separate lemmas.

\begin{lemma}[Parseval's Theorem in \cite{Reference 2}]\label{boolean square sum 1}
Let $f$ be a Boolean function on $Q_n$. Then 

$$ \sum_{S \subseteq [n]} \hat{f}(S)^2 = 1 . $$ 

\end{lemma} 

\begin{lemma}[Exercise 1.11(b) in \cite{Reference 2}]\label{k-function coefficients}
Let $f \in V_k$ be a Boolean function, where $k \geq 1$. Then $\hat{f}$ is $\frac{1}{2^{k-1}} \mathbb{Z}$-valued.  
\end{lemma}

\section{Results}\label{res.}

In this section we prove Theorems \ref{est.} and \ref{stab.}. An outline of the proof is as follows. We first obtain a criterion for a Boolean function to be a $k$-function in terms of its Fourier decomposition which will be used throughout the rest of the paper. We then prove Theorem \ref{stab.}. Next, we  prove a symmetry of $F(n,k)$ and $G(n,k)$ which explains the appearance of $m=min(k,n-k)$ in Theorem \ref{est.}. \newline

We then show how to obtain a $(k+1)$-function on $Q_{n+2}$ given a pair of $k$-functions on $Q_n$, which is the key to proving the lower bound in Theorem \ref{est.}. Next, we introduce a new function which counts the number of ways of writing a non-negative integer as a sum of squares and obtain an upper bound for $F(n,k)$ in terms of this function. We then prove an upper bound for this new function. Finally, we put all our previous results together to prove Theorem \ref{est.}.

\subsection{Criterion for a Boolean function to be a $k$-function}

In \cite{Reference 1}, Gross and Grupel show that a Boolean function $f$ on $Q_n$ is an $n/2$-function if and only if $f \in V_{n/2}$ (Proposition 3.5.). The following lemma uses the same argument to  generalise this result. 

\begin{lemma}\label{Criterion Lemma}
A Boolean function $f$ on $Q_n$ is a $k$-function if and only if $f \in V_k$.
\end{lemma}

\begin{proof}
Consider the linear map $\alpha: V \to V$ given by 

$$ (\alpha f)(x) = \sum_{y \in \Gamma(x)} f(y) .  $$ 

We claim that a Boolean function $f$ on $Q_n$ is a $k$-function if and only if $\alpha f = (n-2k)f$. To see this, for each $x \in Q_n$, let $k(x)=|\{y \in \Gamma(x) : f(x) \neq f(y)\}|$. Then

$$ (\alpha f)(x) = \sum_{y \in \Gamma(x)} f(y) = (n-k(x))f(x)+k(x)(-f(x)) = (n-2k(x)) f(x)  .$$

Then, by definition, $f$ is a $k$-function if and only if $k(x)=k$ for all $x \in Q_n$, i.e. if and only if $\alpha f=(n-2k)f$. \newline

For each $S \subseteq [n]$, since $\chi_{S}$ is an $|S|$-function, we thus have $\alpha \chi_{S} = (n-2|S|) \chi_{S}$. So the Fourier basis diagonalises $\alpha$. Hence, $f$ is a $k$-function if and only if $\alpha f = (n-2k)f$, which happens if and only if $f \in V_k$.  
      
\end{proof}

Throughout the rest of the paper we will view Lemma~\ref{Criterion Lemma} as the definition of a $k$-function.

\subsection{Proof of Theorem \ref{stab.} }

We now prove Theorem \ref{stab.}.

\begin{theo}
Let $k \geq 0$ be an integer. Then the sequence $(G(n,k))_{n=k}^{\infty}$ is increasing and eventually constant. Moreover, denoting by n(k) the first value of $n$ after which the sequence is constant, we have 

$$ 3 \cdot 2^{k-1} - 2 \leq n(k) \leq 4.394 \cdot 2^k . $$

\end{theo}

\begin{proof}

Gross and Grupel note in \cite{Reference 1} that given integers $n \geq m \geq k$ we can think of a $k$-function on $Q_m$ as a $k$-function on $Q_n$ for which all the indices $m < i \leq n$ are irrelevant (Observation 4.4.). Combining this observation with Lemma~\ref{m,n isom.}, we obtain that the sequences $(F(n,k))_{n=k}^{\infty}$ and $(G(n,k))_{n=k}^{\infty}$ are increasing. \newline

Moreover, in light of this observation and Lemma~\ref{m,n isom.}, a moment's thought shows that for all integers $N \geq k$ the following two statements are equivalent:

\begin{itemize}
  \item $G(n,k)=G(N,k)$ for all integers $n \geq N$.
  \item Every $k$-function has at most $N$ relevant indices. 
\end{itemize}

Wellens proved in \cite{Reference 3} that every $k$-function has at most $4.394 \cdot 2^k$ relevant indices (Theorem 1.1.). Hence $(G(n,k))_{n=k}^{\infty}$ is eventually constant and $n(k) \leq 4.394 \cdot 2^k$. In \cite{Reference 4}, Chiarelli, Hatami and Saks recursively construct $k$-functions with $3 \cdot 2^{k-1} - 2$ relevant indices (Theorem 3.1.). Hence $n(k) \geq 3 \cdot 2^{k-1} - 2$.   

\end{proof}

\subsection{Symmetry of $F(n,k)$ and $G(n,k)$}

We now prove a symmetry of $F(n,k)$ and $G(n,k)$.

\begin{lemma}\label{Symmetry Lemma}

Let $0 \leq k \leq n$ be integers. Then $F(n,k)=F(n,n-k)$ and $G(n,k)=G(n,n-k)$.

\end{lemma}

\begin{proof}
 
Define a linear map $\beta: V \to V$ by $\beta f = \chi_{[n]} f$. Since $\chi_{[n]}^2 = 1$, $\beta^2 = id$, where $id$ is the identity function on $V$. We have $\beta \chi_{S} = \chi_{[n]} \chi_{S} = \chi_{S^C}$ for all $S \subseteq [n]$. Hence $\beta$ swaps $V_k$ and $V_{n-k}$. Note that $\beta f$ is a Boolean function if and only if $f$ is. Hence, by Lemma~\ref{Criterion Lemma}, $\beta$ induces a bijection between $k$-functions and $(n-k)$-functions on $Q_n$. So $F(n,k)=F(n,n-k)$. \newline

To show that $G(n,k)=G(n,n-k)$ it suffices to check that $\beta$ respects isomorphisms. Let $f,g \in V$ be isomorphic, say $f = \epsilon g \circ \phi$, where $\epsilon \in \{-1,1\}$ and $\phi \in Aut(Q_n)$. Then 

$$\beta f = \beta (\epsilon g \circ \phi) = \epsilon \chi_{[n]} (g \circ \phi) = \epsilon (\chi_{[n]} \circ \phi^{-1} \circ \phi) (g \circ \phi) = \epsilon ((\chi_{[n]} \circ \phi^{-1}) g ) \circ \phi \ . $$

But $\chi_{[n]} \circ \phi^{-1} = \pm \chi_{[n]}$, so

$$\beta f = \epsilon ((\chi_{[n]} \circ \phi^{-1}) g ) \circ \phi = \pm (\chi_{[n]} g ) \circ \phi = \pm  (\beta g) \circ \phi , $$

so $\beta f$ is isomorphic to $\beta g$.

\end{proof}

Lemma~\ref{Symmetry Lemma} is the reason we extended the notion of isomorphism. It doesn't hold if we use the definition of Gross and Grupel. For example, % by Lemma~\ref{Criterion Lemma} (or directly from the original combinatorial definition), % 
the only $0$-functions are $\pm 1$ and the only $n$-functions are $\pm \chi_{[n]}$. The first two are not isomorphic under the definition of Gross and Grupel while the last two are (for $n \geq 1$), so under the definition of Gross and Grupel, $G(n,0) = 2 > 1 = G(n,n)$ (for $n \geq 1$). Under our definition of isomorphism the first and last two are isomorphic, so $G(n,0)=1=G(n,n)$ (for all $n$). \newline

In light of Lemma~\ref{Symmetry Lemma}, when proving Theorem \ref{est.}, we may assume that $m=k$, i.e. that $n \geq 2k$.

\subsection{Obtaining a $(k+1)$-function on $Q_{n+2}$ from a pair of $k$-functions on $Q_n$}

The following lemma is the key to proving the lower bound in Theorem \ref{est.}. A similar construction was used in \cite{Reference 4} by Chiarelli, Hatami and Saks to recursively construct $k$-functions with $3 \cdot 2^{k-1}-2$ relevant indices.

\begin{lemma}\label{Construction}
Let $f$ and $g$ be $k$-functions on $Q_n$. Then

$$ h = \left(\frac{f+g}{2} \right) x_{n+1} + \left(\frac{f-g}{2} \right) x_{n+2} $$

is a $(k+1)$-function on $Q_{n+2}$.  
\end{lemma}

\begin{proof}
By Lemma~\ref{Criterion Lemma}, we need to check that $h$ is a Boolean function in $V_{k+1}$. By considering the four possible values for the pair $(x_{n+1},x_{n+2})$, we see that the values obtained by $h$ are those obtained by $\pm f$ and $\pm g$. Since $f$ and $g$ are Boolean functions, so is $h$. Since $f$ and $g$ are in $V_k$, so are $(f+g)/2$ and $(f-g)/2$. Hence $\left((f+g)/2\right)  x_{n+1}$ and $\left((f-g)/2\right)  x_{n+2}$ are in $V_{k+1}$, since $(f+g)/2$ and $(f-g)/2$ are functions on $Q_n$. Hence $h \in V_{k+1}$.  
    
\end{proof}

\begin{cor}\label{Squaring}

Let $0 \leq k \leq n$ be integers. Then $F(n+2,k+1) \geq F(n,k)^2$.

\end{cor}

\begin{proof}
In Lemma~\ref{Construction}, distinct pairs $(f,g)$ give distinct $h$.

\end{proof}

\begin{cor}\label{LBF}
Let $k \geq 0$ be an integer. Then $F(2k,k) \geq 2^{2^k}$.
\end{cor}

\begin{proof}
This follows from $F(0,0)=2$ and iterating Corollary~\ref{Squaring} with $n=2k$.

\end{proof}

Lemma~\ref{Construction} gives a way of constructing a $(k+1)$-function given a pair of $k$-functions. One might ask whether every $(k+1)$-function arises in this way. It turns out this is not the case. We give an example of a $4$-function which cannot be obtained from a pair of $3$-functions in this way. Note that the $h$ in Lemma~\ref{Construction} is ``covered'' by  the indices $n+1$ and $n+2$, in the sense that for all $S \subseteq [n+2]$ with $\hat{h}(S) \neq 0$, either $n+1 \in S$ or $n+2 \in S$. Hence it is sufficient to construct a $4$-function $h$ which cannot be covered by two indices. \newline

We have $1$-functions $x_1$ and $x_2$, so by Lemma~\ref{Construction}, we have a $2$-function

$$ f(x_1,x_2,x_3,x_4) = \left(\frac{x_1+x_2}{2} \right) x_3 + \left(\frac{x_1-x_2}{2} \right) x_4 = \frac{x_1 x_3 + x_2 x_3 + x_1 x_4 - x_2 x_4}{2} .  $$ 

Let $g_1, g_2, g_3$ and $g_4$ be copies of $f$ with disjoint relevant indices. Let 

$$ h = f(g_1,g_2,g_3,g_4) = \frac{g_1 g_3 + g_2 g_3 + g_1 g_4 - g_2 g_4}{2} . $$   

Then it is easy to check that $h$ is a $4$-function with $64$ non-zero terms in its Fourier decomposition and that for every relevant index $i$ there are precisely $16$ sets $S$ with 
$i \in S$ and $\hat{h}(S) \neq 0$. Hence $h$ cannot be covered by $2$ indices.

\subsection{Relation between $F(n,k)$ and $S(q,t)$}

We now introduce a new function, $S(q,t)$, and prove an upper bound for $F(n,k)$ in terms of $S(q,t)$. For integers $q,t \geq 0$, let $S(q,t)$ denote the number of
$x \in \mathbb{Z}^t$ such that

$$ \sum_{i=1}^t x_i^2 = q .$$    

We then have the following lemma.

\begin{lemma}\label{F,S Lemma}
Let $1 \leq k \leq n$ be integers. Then $F(n,k) \leq S\left(4^{k-1},\binom{n}{k}\right)$.
\end{lemma}

\begin{proof}
Let $f$ be a $k$-function on $Q_n$. By Lemma~\ref{k-function coefficients} and Lemma~\ref{Criterion Lemma}, $\hat{f}$ is $\frac{1}{2^{k-1}} \mathbb{Z}$-valued, so write $\hat{f}(S) = \frac{x_S}{2^{k-1}}$, where $x_S \in \mathbb{Z}$, for all $S \subseteq [n]$ with $|S|=k$. Note that $\hat{f}(S) = 0$ for $S \subseteq [n]$ with $|S| \neq k$  by Lemma~\ref{Criterion Lemma}. By Lemma~\ref{boolean square sum 1},

$$ \sum_{S \in \binom{[n]}{k}} x_S^2 = 4^{k-1} .$$

Distinct $f$ give distinct $x \in \mathbb{Z}^{\binom{[n]}{k}}$, so the result follows.   

\end{proof}

\subsection{An upper bound for $S(q,t)$} 

The function $S(q,t)$ has been studied in number theory, where it is denoted by $r_t(q)$. The author searched the literature but was only able to find estimates
in the regime where $t$ is fixed and $q$ is large, whereas for our purposes we need to consider the regime where both $q$ and $t$ are large and $t$ is 
much larger than $q$. When $t$ is much larger than $q$, most of the $x_i$ have to be 0, so the size of $S(q,t)$ is governed less by the number theory and more by
the combinatorics of choosing which $x_i$ are non-zero. We have the following upper bound for $S(q,t)$.

\begin{lemma}\label{UBS}
For all integers $t \geq q \geq 0$, we have 

$$S(q,t) \leq 2^{q \log_2 t + O(q \log_2 q)} .$$    
\end{lemma}     

\begin{proof}
We first prove an upper bound for $S(q,t)$ for all integers $q,t \geq 0$. If $x \in \mathbb{Z}^t$ is such that $ \sum_{i=1}^t x_i^2 = q$, then $|x_i| \leq \sqrt{q}$ for all $i \in [t]$, so there are at most $2\sqrt{q}+1$ possibilities for each $x_i$. Hence $S(q,t) \leq \left(2\sqrt{q}+1\right)^t$. Now suppose $t \geq q \geq 0$ are integers. For each $x \in \mathbb{Z}^t$ with $ \sum_{i=1}^t x_i^2 = q$, the set $\{i \in [t] : x_i \neq 0\}$ has size at most $q$, so we can pick a subset of $[t]$ of size $q$ containing it. Then there are $\binom{t}{q}$ such subsets and for each subset there are at most $S(q,q)$ different $x \in \mathbb{Z}^t$ with $ \sum_{i=1}^t x_i^2 = q$ for which that subset is picked, so $S(q,t) \leq \binom{t}{q} S(q,q)$. Combining these two bounds, we have

$$ S(q,t) \leq \binom{t}{q} S(q,q) \leq \binom{t}{q} \left(2\sqrt{q}+1\right)^q = 2^{q \log_2 t + O(q \log_2 q)} . $$       

For the last equality, note that we have

$$ \frac{t^q}{q^q} \leq \binom{t}{q} \leq \frac{t^q}{q!} $$

for all integers $t \geq q \geq 0$ and hence

$$ \binom{t}{q} = 2^{q \log_2 t + O(q \log_2 q)} . $$

\end{proof}

By considering $x \in \mathbb{Z}^t$ with $x_i \in \{-1,1\}$ for $q$ different $i \in [t]$ and $x_i = 0$ for all other $i$, we have

$$ S(q,t) \geq \binom{t}{q} 2^q = 2^{q \log_2 t + O(q \log_2 q)} $$

for all integers $t \geq q \geq 0$. Hence the bound in Lemma~\ref{UBS} is tight.

\subsection{Proof of Theorem \ref{est.}}

We now put all our previous results together to prove Theorem \ref{est.}. 

\begin{thr}

Let $0 \leq k \leq n$ be integers. Then

$$2^{2^{m+o(1)}} \leq  G(n,k) \leq 2^{2^{2m+O\left(\log_2 m\right)}} , $$

where $m=min(k,n-k)$.

\end{thr}

\begin{proof}

By Lemma~\ref{Symmetry Lemma}, we may assume that $m=k$, i.e. that $n \geq 2k$. We first prove the lower bound. We have

\begin{align*}
G(n,k) &\geq G(2k,k) & &\text{(by Theorem \ref{stab.})} \\
&\geq \frac{F(2k,k)}{2^{2k+1}(2k)!} & &\text{(by Lemma \ref{F,G Lemma})} \\
&\geq \frac{2^{2^k}}{2^{2k+1}(2k)!} = 2^{2^{k+o(1)}} . & &\text{(by Corollary \ref{LBF})}  
\end{align*}

We now prove the upper bound. We have

\begin{align*}
G(n,k) &\leq G(n(k),k) & &\text{(by Theorem \ref{stab.})} \\
&\leq F(n(k),k) & &\text{(by Lemma \ref{F,G Lemma})} \\
&\leq S\left(4^{k-1},\binom{n(k)}{k}\right) & &\text{(by Lemma \ref{F,S Lemma})} \\
&\leq 2^{2^{2k + O\left(\log_2 k\right)}} . & &\text{(by Lemma \ref{UBS} and Theorem \ref{stab.})}  
\end{align*}

\end{proof}

\section{Acknowledgement}

The author thanks Robert Johnson for suggesting the problem and for useful feedback on a draft of this paper and the anonymous referee for helpful suggestions. This work was supported by the Engineering and Physical Sciences Research Council.


\begin{thebibliography}{99}

\bibitem{Reference 1} Renan Gross and Uri Grupel. \textit{Indistinguishable sceneries on the Boolean hypercube.} Combinatorics, Probability and Computing, Volume 28, Issue 1, January 2019, pp. 46 - 60.
 
\bibitem{Reference 2} Ryan O'Donnell. \textit{Analysis of Boolean Functions.} Cambridge University Press, New York, NY, USA, 2014. 

\bibitem{Reference 3} Jake Wellens. \textit{Relationships between the number of inputs and other complexity measures of Boolean functions.} arXiv preprint, arXiv:2005.00566, May 2020.

\bibitem{Reference 4} John Chiarelli, Pooya Hatami and Michael Saks. \textit{An Asymptotically Tight Bound on the Number of Relevant Variables in a Bounded Degree Boolean function.} Combinatorica 40, 237–244, March 2020.

\bibitem{Reference 5} Peter van Hintum. \textit{Biased Partitions of $\mathbb{Z}^n$.} European Journal of Combinatorics, Volume 79, 2019, Pages 262-270.

\bibitem{Reference 6} Itai Benjamini, and Harry Kesten. \textit{Distinguishing sceneries by observing the scenery along a random walk path.} Journal d’Analyse Mathématique, Volume 69, 1996, pages 97–135. 

\bibitem{Reference 7} Hilary Finucanea, Omer Tamuza and Yariv Yaaria. \textit{Scenery Reconstruction on Finite Abelian Groups.} Stochastic Processes and their Applications, Volume 124.8, August 2014, pp.2754-2770.

\bibitem{Reference 8} Elon Lindenstrauss. \textit{Indistinguishable sceneries.} Random Structures and Algorithms, Volume 14.1, January 1999, pp. 71-86.




\end{thebibliography}
\end{document}